\documentclass[12pt]{amsart}
\usepackage{amsfonts, amssymb, amsgen, amsbsy, amstext, amsopn, amsmath,
latexsym, indentfirst,color,longtable,paralist}
\usepackage{txfonts}
\usepackage{enumitem}
\usepackage{centernot}
\usepackage[utf8]{inputenc}
\usepackage{verbatim} 

\usepackage{hyperref} 

\usepackage[american]{babel}

\usepackage{fge} 

\newcommand{\B}{\mathbb{B}}

\newcommand{\N}{\mathbb{N}}

\newcommand{\Z}{\mathbb{Z}}

\newcommand{\cC}{\mathcal{C}}
\newcommand{\cD}{\mathcal{D}}
\newcommand{\cE}{\mathcal{E}}
\newcommand{\cF}{\mathcal{F}}
\newcommand{\cG}{\mathcal{G}}
\newcommand{\cH}{\mathcal{H}}

\newcommand{\cP}{\mathcal{P}}

\newcommand{\cS}{\mathcal{S}}

\newcommand{\cT}{\mathcal{T}}

\newcommand{\cd}{\mathfrak{d}}

\newcommand{\cs}{\mathfrak s}

\newcommand{\ep}{\epsilon}

\newcommand{\sm}{\setminus}

\newcommand{\ls}{\mbox{\large $($}}
\newcommand{\rs}{\mbox{\large $)$}}

\newcommand{\diams}{\mbox{\rm diam}\;\!}
\newcommand{\diam}{\mbox{\rm diam}}

\newcommand{\res}{\mbox{\LARGE{$\llcorner$}}}

\newcommand{\eR}{{\overline {\mathbb R}}}

\newcommand{\beqas}{\begin{eqnarray*}}
	\newcommand{\eeqas}{\end{eqnarray*}}
\newcommand{\beqa}{\begin{eqnarray}}
	\newcommand{\eeqa}{\end{eqnarray}}
\newcommand{\beq}{\begin{equation}}
	\newcommand{\eeq}{\end{equation}}
\newcommand{\bce}{\begin{center}}
	\newcommand{\ece}{\end{center}}

\newcommand{\set}[1]{\left\{ #1 \right\}}            

\newtheorem{The}{Theorem}[section]
\newtheorem{Lem}[The]{Lemma}
\newtheorem{Def}{Definition}
\newtheorem{Rem}{Remark}
\newtheorem{Pro}[The]{Proposition}
\newtheorem{Cor}[The]{Corollary}
\newtheorem{Exa}{Example}
\newtheorem{Con}{Conjecure}

\newcommand{\bt}{\begin{The}}
	\newcommand{\et}{\end{The}}
\newcommand{\bl}{\begin{Lem}}
	\newcommand{\el}{\end{Lem}}
\newcommand{\bd}{\begin{Def}\rm}
	\newcommand{\ed}{\end{Def}}
\newcommand{\br}{\begin{Rem}\rm}
	\newcommand{\er}{\end{Rem}}
\newcommand{\bpr}{\begin{Pro}}
	\newcommand{\epr}{\end{Pro}}
\newcommand{\bc}{\begin{Cor}}
	\newcommand{\ec}{\end{Cor}}
\newcommand{\bj}{\begin{Con}}
	\newcommand{\ej}{\end{Con}}
\newcommand{\bex}{\begin{Exa}}
	\newcommand{\eex}{\end{Exa}}

\newcommand{\Sr}{\mathcal{S}_{\mu,\zeta}}
\renewcommand{\sm}{\mathbin{\fgebackslash}}
\newcommand{\mysetminus}{\mathbin{\fgebackslash}}

\newcommand{\psir}{\psi_{\zeta}}
\newcommand{\RADot}{\mbox{\LARGE{$\llcorner$}}}

\numberwithin{equation}{section}

\begin{document}

\title
[A study of measure-theoretic area formulas]
{{\bf A study of measure-theoretic area formulas}}

\author{Giacomo Maria Leccese}
\address{Giacomo Maria Leccese, 
	SISSA, Via Bonomea 265, 34136 Trieste, Italy}
\email{giacomomaria.leccese@sissa.it}
\author{Valentino Magnani}
\address{Valentino Magnani, Dip.to di Matematica, Universit\`a di Pisa \\
Largo Bruno Pontecorvo 5 \\ I-56127, Pisa}
\email{valentino.magnani@unipi.it}

\date{\today}
\thanks{The author acknowledges the support of the University of Pisa, Project PRA 2018 49.}

\begin{abstract}
	We present a complete study of measure-theoretic area formulas in metric spaces, providing different measurability conditions.
\end{abstract}

\maketitle

\tableofcontents

\section{Introduction}
 The problem of finding a proper notion of surface area has a long history in Calculus of Variations and in Geometric Measure Theory. Hausdorff measures, or more generally Carath\'eodory measures, are widely accepted as natural notions of surface area in metric spaces, \cite[$\!\mathsection\!$ 2.10.1]{Federer69}. On the other side, several important results of Geometric Measure Theory in Euclidean spaces are based on the fundamental concept of {\em rectifiable set}, whose Hausdorff measure can be computed by the
classical area formula. Rectifiable sets can be considered in metric spaces, \cite[$\!\mathsection\!$ 3.2.14]{Federer69}, and in relation to their density properties, a general metric area formula was established in \cite{Kir94}, see also \cite{AmbKir2000Rect} for a different approach.

Substantial obstructions appear when we wish to compute the Hausdorff measure (or the spherical measure) of {\em purely unrectifiable sets}. An important family of such sets naturally appears in the geometries of homogeneous groups, \cite{FS82}, where also smooth submanifolds can be purely unrecti\-fiable. This is a consequence of their Hausdorff dimension, that can be strictly greater than their topological dimension, \cite[$\!\mathsection\!$ 0.6.B]{Gromov1996}.  In this framework, the use of {\em measure-theoretic area formulas} seems to be the only approach to find integral formulas for the Hausdorff and spherical measure of submanifolds. In this connection, we mention some recent papers \cite{FSSC15,JNGV20prArea,Mag31,Magnani2019Area,MST18}.

If the spherical measure is replaced by the centered Hausdorff measure, additional measure-theoretic area formulas have been proved in \cite{FSSC15}, leading to the relationship between perimeter measure and centered Hausdorff measure in stratified groups. In sum, applications to area formulas for submanifolds and to the perimeter measure in classes of homogeneous groups represent the first motivation of the present work. However, taking into account the deep relationship between density and rectifiability, \cite{Preiss87,PreissTiser92}, there are also interesting connections with the so-called Besicovitch $\frac12$-problem and rectifiability, discussed in \cite{Mag30,Mag17Dens}.
Considering that measure-theoretic area formulas are stated in general metric spaces, we would not be surprised about further applications.

We continue the study of measure-theoretic area formulas, according to which a suitable 
integration formula can be found between a Borel measure  and a Carathéodory measure, \cite{Mag30}. The latter measure $\psi_\zeta$ is obtained by a gauge $\zeta:\cS\to[0,+\infty]$
defined on a class of subsets $\cS\subset\cP(X)$ of a metric space $X$.
When $\zeta$ is proportional to some real power $\alpha>0$ of the diameter of a set, the Carathéodory measure coincides with the Hausdorff measure $\cH^\alpha$. In analogous way we obtain the spherical measure $\cS^\alpha$ (Definition~\ref{d:sizephi}).

The general measure-theoretic area formula \cite[Theorem~9]{Mag30} essentially arises from two key results: \cite[Theorem~2.10.17(2)]{Federer69} and \cite[Theorem~2.10.18(1)]{Federer69}. On the other hand, in considering a general $\zeta$, it may happen that 
the quotients $\mu(S)/\zeta(S)$ in the upper covering limits 
are not well defined for some $S\in\cS$ of possibly small diameter.
The values of $\zeta(S)$ may either vanish or be equal to $+\infty$ and the same values may be taken by $\mu(S)$. Another point in \cite[Theorem~2.10.18(1)]{Federer69} is that the gauge $\zeta$ need not be necessarily defined on the enlargement $\hat S$ of $S$, see \eqref{d:tau-enlargement}.

Revised versions of \cite[Theorem~2.10.17(2)]{Federer69} and \cite[Theorem~2.10.18(1)]{Federer69}
are stated in \cite{Mag30} without proof. Actually, all the results of \cite{Mag30} are only stated. The previous two theorems can be extended to our Lemmas~\ref{lemma minor} and \ref{lemma major}, hence leading in turn to two more general measure-theoretic area formulas, proved in Theorems~\ref{the:meastheoarea} and \ref{theo-mea area form}. All of these results provide a reasonably complete picture of measure-theoretic area formulas for Carathéodory measures in metric spaces, giving indeed the second motivation for this work.

To define the Federer density, we introduce the quotient function $Q_{\mu,\zeta}$ with respect to a measure $\mu$ over $X$ and a gauge $\zeta$, (Definition~\ref{def:FedererDensity}). We restrict $Q_{\mu,\zeta}$ to
a {\em subclass $\Sr\subset\cS$} where it is unambiguously defined. It is somehow surprising the fact that, despite this restriction, the Federer density $F^\zeta(\mu,x)$, defined in \eqref{eq federer den} by $Q_{\mu,\zeta}$, still yields the following measure-theoretic area formula
\begin{equation}\label{eq:mestheoarea-intro}
	\mu(B)=\int_B F^\zeta(\mu,x)\,d\psi_\zeta(x).
\end{equation}
The cases where $\psi_\zeta$ is either the Hausdorff measure or the spherical measure is important for applications, where the Federer density has also more chances to be explicitly found.
This fact has been shown for submanifolds or intrinsic submanifolds in different classes of noncommutative homogeneous groups, \cite{FSSC15,JNGV20prArea,Mag31,Magnani2019Area,MST18}.

The Federer density $F^\zeta(\mu,\cdot)$ is denoted by $\cd^\alpha(\mu,\cdot)$ for the Hausdorff measure and by $\cs^\alpha(\mu,\cdot)$ for the spherical measure. In these cases it is also important to show that $\cd^\alpha(\mu,\cdot)$ and $\cs^\alpha(\mu,\cdot)$ are measurable (or Borel), but this is not straightforward. We mention that when the metric space is a homogeneous group and a suitable restriction of $\psi_\zeta$ is assumed to be locally doubling, the measurability of $\cd^\alpha(\mu,\cdot)$ and $\cs^\alpha(\mu,\cdot)$ is proved in \cite[Proposition~2.3]{JNGV20prArea}.
It is possible to show that in a metric space the Federer density $\cd^\alpha(\mu,\cdot)$ is Borel with respect to the subspace topology of a fixed set, according to Theorem~\ref{th:Hmeasurability}.
By a mild continuity assumption on the diameter function (Definition~\ref{diam reg}) 
and some more work, the previous Borel measurability can be extended to $\cs^\alpha(\mu,\cdot)$, see Theorem~\ref{th:spher Borel}. 
Important consequences are the versions of \eqref{eq:mestheoarea-intro} for the Hausdorff and spherical measure, established in Theorems~\ref{area formula haus Borel}, \ref{area formula haus}, \ref{th:area formula sph Borel} and \ref{area formula spher meas}.

\section{Basic notions}

Throughout the paper, $X$ denotes a metric space equipped with a distance $d$.
A {\em covering relation} is a subset $\cC$ of $\{(x,S): x\in S\in\cP(X)\}$. 
Defining for $A\subset X$ the class  
\[
\cC(A)=\{S: x\in A,\,(x,S)\in\cC\},
\] 
we say that $\cC$ is {\em fine at $x$}, if for every $\ep>0$ there exists $S\in\cC(\{x\})$ such that $\diams S<\ep$.
Following the terminology of \cite[$\!\mathsection\!$ 2.8.1]{Federer69}, a nonempty family of sets
$\cF\subset\cP(X)$ {\em covers $A\subset X$ finely}, if for each $a\in A$
and $\ep>0$ there exists $S\in\cF$ with $a\in S\in \cF$ 
such that $\diam(S)<\ep$.
According to \cite[$\!\mathsection\!$ 2.8.16]{Federer69}, the notion of covering relation yields the following notion of ``upper'' and ``lower covering limits''.
%
%
%
%
%
\begin{Def}[Covering limits]\label{def:coveringlimits} \rm
	If $\cC$ is a covering relation which is fine at $x\in X$, 
	\[\cC(\{x\})\subset \cD\subset\cC(X)\]
	 and $f:\cD\to\eR$, then we define the {\em upper and lower covering limit}, respectively as
	\begin{eqnarray}
		&&(\cC)\limsup_{S\to x}f=\inf_{\ep>0}\sup\{f(S): S\in\cC(\{x\}), \diams S<\ep\}\,,\\ 
		&&(\cC)\liminf_{S\to x}f=\sup_{\ep>0}\inf\{f(S): S\in\cC(\{x\}), \diams S<\ep\}\,.
	\end{eqnarray}
\end{Def}
The closed ball and the open ball of center $x\in X$ and radius $r>0$ are denoted by 
	\[ 
	\B(x,r)=\{y\in X: d(x,y)\le r\}\quad\mbox{and}\quad B(x,r)=\{y\in X: d(x,y)< r\}\,,
	\]
respectively. The notion of $\delta$-covering and the Carath\'eodory construction of 
\cite[$\!\mathsection\!$ 2.10.1]{Federer69} are both recalled in the next definition. 

%
%
%
%
\begin{Def}\label{d:sizephi}\rm 
If $\delta>0$, $E\subset X$ and $\cE=\set{E_j:j\in\N}\subset\cP(X)$ is a covering of $E$
such that $\diam(E_j)\le \delta$ for every $j\in\N$, we say that $\cE$ is a 
{\em $\delta$-covering} of $E$.
Let $\cS\subset\cP(X)$ and let $\zeta:\cS\to[0,+\infty]$ be 
a fixed {\em gauge}.
For $R\subset X$, we define
\begin{eqnarray*}
		&&\phi_{\zeta,\delta}(R)=\inf \left\lbrace\sum_{j=0}^\infty \zeta(E_j):
		\set{E_j:j\in\N}\subset\cS\;\text{is a $\delta$-covering of $R$} \right\rbrace\,.
\end{eqnarray*}
	The {\em $\zeta$-approximating measure}, or {\em Carath\'eodory measure}, is defined as follows
	\[\psi_\zeta(A)=\sup_{\delta>0}\phi_{\zeta,\delta}(A)\]
for every $A\subset X$. Denoting by $\cF$ the family of closed sets of $X$, for $\alpha,\,c_\alpha>0$, we define
	$\zeta_\alpha:\cF\to[0,+\infty]$ by
	\begin{equation*}
		\zeta_\alpha(S)= c_\alpha\, \diam(S)^\alpha\,.
	\end{equation*}
	Then the {\em $\alpha$-dimensional Hausdorff measure} is $\cH^\alpha=\psi_{\zeta_\alpha}$. 
We define the set function 
\[
\zeta_{b,\alpha}:\cF_b\to[0,+\infty)
\]
as the restriction of $\zeta_\alpha$ to the family of all closed balls in $X$,
that we denote by $\cF_b$.
Then $\psi_{\zeta_{b,\alpha}}$ is the {\em $\alpha$-dimensional spherical Hausdorff measure}, denoted by $\cS^\alpha$.
\end{Def}
We will use the terminology ``measure over $X$" to indicate an outer measure defined on all subsets of $X$, \cite[$\!\mathsection\!$ 2.1.2]{Federer69}.
A gauge $\zeta:\cS\to[0,+\infty]$ is also fixed on a nonempty class $\cS\subset\cP(X)$.
The next definition gives the notion of Federer density, introduced in \cite[Definition~4]{Mag30}.

%
%
%
%
%
\begin{Def}[Federer density]\label{def:FedererDensity}\rm 
Let $\mu$ be a measure over $X$ and let $\zeta:\cS\to[0,+\infty]$
be a gauge. We introduce the subfamily of sets
	\begin{equation}\label{d:Sr}
	\cS_{\mu,\zeta}=\cS\sm \Big\{S\in\cS\,:\; \zeta(S)=\mu(S)=0\; \mbox{\rm or}\; 
	\mu(S)=\zeta(S)=+\infty\, \Big\}\,,
	\end{equation}
	along with the covering relation 
	\[
	\cC_{\mu,\zeta}=\{(x,S): x\in S\in\cS_{\mu,\zeta}\}.
	\]
	We choose $x\in X$ and assume that $\cC_{\mu,\zeta}$ is fine at $x$. We define the quotient function
	\begin{eqnarray}\label{q}
		Q_{\mu,\zeta}:\cS_{\mu,\zeta}\to[0,+\infty],\quad
		Q_{\mu,\zeta}(S)=\left\{\begin{array}{ll} +\infty & \mbox{if $\zeta(S)=0$} \\ 
			\mu(S)/\zeta(S) & \mbox{if $0<\zeta(S)<+\infty$} \\
			0 & \mbox{if $\zeta(S)=+\infty$}
		\end{array}\right.\,.
	\end{eqnarray}
	Then the {\em Federer density}, or 
	{\em upper $\zeta$-density of $\mu$ at $x\in X$}, is well defined as
	\begin{equation}\label{eq federer den}
		F^\zeta(\mu,x)=(\cC_{\mu,\zeta})\limsup_{S\to x}Q_{\mu,\zeta}(S)\,.
	\end{equation}
\end{Def}
We use a special notation when we consider Federer densities with respect to $\zeta_\alpha$ and $\zeta_{b,\alpha}$, respectively.
If $\mu$ is a measure over $X$ and $\cC_{\mu,\zeta_{b,\alpha}}$ is fine at $x\in X$, then we set
\[
\cs^\alpha(\mu,x)=F^{\zeta_{b,\alpha}}(\mu,x).
\]
If $\cC_{\mu,\zeta_\alpha}$ is fine at $x$, then we set 
\[
\cd^\alpha(\mu,x)=F^{\zeta_\alpha}(\mu,x).
\]
These densities will appear in the measure-theoretic area formulas
for the Hausdorff and spherical measure, see Section~\ref{sect:HausdSpher}. 
\begin{Rem}\rm It is important to compute the upper limit \eqref{eq federer den} only over the sets $S\in\cC_{\mu,\zeta}$, even if the definition of $Q_{\mu,\zeta}$ in \eqref{q} could be extended to any subset.
Let us consider the gauge $\zeta_\alpha:\cF\to[0,\infty]$, with
$\zeta_\alpha(A)=c_\alpha\diam(A)^\alpha$ for every $A\in\cF$, where $\alpha>0$. 
If we extend the definition of $Q_{\mu,\zeta}$ to all of $\cF$ using \eqref{q},
then for every $x\in X$ it holds
\beq\label{eq:FedDensClosed}
(\cF)\limsup_{S\to x}Q_{\mu,\zeta}(S)=+\infty.
\eeq
Indeed, for every $\epsilon>0$, we have 
$$
\sup\{Q_{\mu,\zeta}(S):S\in \cC(\{x\}),\,\diam (S)<\epsilon\}=Q_{\mu,\zeta}(\{x\})=+\infty.
$$
The upper covering limit \eqref{eq:FedDensClosed} now differs from the Federer density $\cd(\mu,x)$. It cannot give the measure-theoretic area formulas of Theorems~\ref{area formula haus Borel} and \ref{area formula haus}.
\end{Rem}
In the sequel, a number $\tau>1$ is fixed and we will use the notion of 
{\em enlargement of $S$}, namely we introduce the set
\begin{equation}\label{d:tau-enlargement}
	\hat{S}=\bigcup\{T\in\Sr:\,T\cap S\ne\emptyset\;\text{and}\;\diam(T)\le\tau\diam(S)\},
\end{equation}
where $\Sr$ is defined in \eqref{d:Sr}.

\section{Preliminary tools}

As mentioned in the introduction, the notion of Federer's density
based on the quotient function \eqref{q} differs from the 
upper covering limits appearing in \cite[Theorem~2.10.17(2)]{Federer69} and 
\cite[Theorem~2.10.18(1)]{Federer69}. The next lemmas show that
in these theorems the Federer density \eqref{eq federer den} 
can replace the upper covering limits.

We will use the notion of {\em regular measure} taken from \cite[$\!\mathsection\!$ 2.1.5]{Federer69}.

%
%
%
%
%
%
%
%
%
%
%
%
\bl\label{lemma minor}
Let $\zeta:\cS\to[0,+\infty]$ be a gauge and let $A\subset X$ be nonempty.
We consider a measure $\mu$ over $X$, $t>0$ and assume that the following conditions hold.
\begin{enumerate}
	\item $\mu$ is a regular measure.
	\item $\cS_{\mu,\zeta}$ covers $A$ finely.
	\item For every $x\in A$ we have $F^\zeta(\mu,x)<t$.
\end{enumerate}
Then $\mu(E)\le t\psi_\zeta(E)$ for all $E\subset A$. 
\el

\begin{proof}
	Let $B(\mu,t,\delta)$ be the set of all those points $x\in E$ such that whenever $S\in\cS_{\mu,\zeta}$
	and $x\in S$ with $\diam (S)\le\delta$, we have
	\begin{equation*}
		\mu(E\cap S)\le t\zeta(S).
	\end{equation*}
	We claim that $E=\bigcup_{k=0}^\infty B(\mu,t,2^{-k})$. The hypothesis implies that for every fixed $x\in E$, there exists $\bar k\in\mathbb{N}$ such that for every $S\in\cS_{\mu,\zeta}$ containing $x$ and $\diam(S)\le 2^{-\bar k}$, we have 
	\[
	Q_{\mu,\zeta}(S)<t.
	\]
There are two possible cases: $\zeta(S)=+\infty$ or
	$$0<\zeta(S)<+\infty\hspace{1cm}\text{and}\hspace{1cm}\frac{\mu(E\cap S)}{\zeta(S)}\le\frac{\mu(S)}{\zeta(S)}\le t.$$
	In any case $x\in B(\mu,t,2^{-\bar k})$ and the claim is proved. 
	We notice that 
	$$\mu\left(B\left(\mu,t,2^{-k}\right)\right)\le t\phi_{\zeta,2^{-k}}\left(B\left(\mu,t,2^{-k}\right)\right).$$
	Indeed for every family $\left\{E_i\right\}_{i\in\mathbb{N}}\subset\mathcal{S}$ such that $\diam(E_i)\le2^{-k}$ and 
	\[
	B\left(\mu,t,2^{-k}\right)\subset\bigcup_{i\in\mathbb{N}}E_i,
	\]
	we have $$\mu\left(B\left(\mu,t,2^{-k}\right)\right)  \le
	\sum_{i\in\mathbb{N}}\mu\left(B\left(
	\mu,t,2^{-k}\right)\cap E_i\right)   \le
	\sum_{i\in\mathbb{N}}\mu\left(E\cap E_i\right)    \le
	t\sum_{i\in\mathbb{N}}\zeta\left(E_i\right).$$ 
	The last inequality holds in the case $E_i\in \cS_{\mu,\zeta}$ for every $i$, where 
	we use the definition of $B\left(\mu,t,2^{-k}\right)$. It is also true for every $E_i\in\mathcal{S}\sm\cS_{\mu,\zeta}$, since in this case if $\mu\left(E_i\right)=\zeta\left(E_i\right)=+\infty$ the inequality is obvious, otherwise $\mu\left(E\cap E_i\right)\le\mu\left( E_i\right)=\zeta\left(E_i\right)=0$.
	It follows that
	\begin{equation*}
		\begin{split}
			\mu(E)&=\lim_{k\to\infty}\mu\left(B\left(\mu,t,2^{-k}\right)\right)\le t\limsup_{k\to\infty}\phi_{\zeta,2^{-k}}\left(B\left(\mu,t,2^{-k}\right)\right)\\
			&\le t\limsup_{k\to\infty}\psi_\zeta(B(\mu,t,2^{-k}))\le
			t\psi_\zeta(E),
		\end{split}
	\end{equation*}
	where in the first equality we have used the assumption that $\mu$ is a regular measure.
\end{proof}

%
%
%

%
%
%
%
%
%
%
%
%
%
%
%
%
%

An aspect to be clarified in \cite[Theorem~2.10.18(1)]{Federer69}
is that the gauge $\zeta$ in this theorem may not be necessarily defined on 
the enlargement $\hat S$.
In the next lemma, the conditions \eqref{eq:StildeConditions} take into account this fact
and the upper covering limit in \cite[Theorem~2.10.18(1)]{Federer69} is replaced by the Federer density.

\bl\label{lemma major}
	Let $\mu$ be a measure over $X$, $B\subset X$ is nonempty and 
	let $\zeta:\cS\to[0,+\infty]$ be a gauge. 
	We assume that the following conditions hold.
	\begin{enumerate}
		\item $\Sr$ is made by closed and $\mu$-measurable sets.
		\item $\Sr$ covers $B$ finely.
		\item We have $c\ge1$ and $\eta>0$ such that for every $S\in \Sr$
		there exists $\tilde S\in\cS$ such that
		\begin{equation}\label{eq:StildeConditions}
			\hat S\subset\tilde S,\hspace{1cm}\diam(\tilde S)\le c\, \diam (S)\hspace{1cm}and\hspace{1cm}\zeta(\tilde S)\le\eta\zeta(S),
		\end{equation}
	where $\hat S$ is the enlargement of $S$ defined in \eqref{d:tau-enlargement}.
	\item Let $t>0$ be such that  
		\[F^\zeta(\mu,x)>t\]
		for all $x\in B$. 
	\end{enumerate}
	Then for any open set $V$ containing $B$, we have $t\psi_\zeta(B)\le\mu(V)$.
\el

\begin{proof}
	It is obviously not restrictive to assume $\mu(V)<+\infty$. For each $\delta>0$, 
	we define
		$$\cT_{\delta}=\left\{S\in \Sr:t\zeta(S)<\mu(S),\,S\subset V\;\text{and}\;\text{diam}( S)\le\frac{\delta}{c}\right\}.$$
	We claim that $\cT_\delta$ covers $B$ finely.
	We fix $x\in B$. Since $F^\zeta(\mu,x)>t$ and $V$ is open, there exists $S\in \Sr$ such that $S\subset V$, 
	\[
	\diam(S)<\frac{\delta}{c}\quad\text{and}\quad Q_{\mu,\zeta}(S)>t.
	\]
	Since $\zeta(S)<+\infty$, we have $\zeta(S)=0$ and $\mu(S)>0$, or
$\zeta(S)>0$ and $Q_{\mu,\zeta}(S)=\mu(S)/\zeta(S)$. In both cases we have proved that $t\zeta(S)<\mu(S)$, so that the claim is established.
In view of \cite[Corollary~2.8.6]{Federer69}, there exists a disjoint subfamily 
$\mathcal {G}$ of $\cT_\delta$ such that 
	\begin{equation*}
		B\mysetminus\bigcup \mathcal {H}\subset\bigcup\left\{\hat S:S\in \mathcal {G}\mysetminus \mathcal {H}\right\}
	\end{equation*}
for every finite class $\mathcal {H}\subset \mathcal {G}$. Since $\cup \mathcal {G}\subset V$ and $\mu(V)<+\infty$, the class $\cG$ must be countable.
Thus, given $\epsilon>0$, we can choose $\mathcal {H_\epsilon}\subset\cG$ so that
\begin{equation*}
		\sum_{S\in \mathcal {G}\mysetminus \mathcal {H_\epsilon}} \mu(S)<\epsilon.
\end{equation*}
We now use the hypothesis (3), choosing some $\tilde S\in\cS$ such that
conditions \eqref{eq:StildeConditions} hold for every $S\in \cG\sm\cH_\epsilon$.
Taking into account that $\mathcal {H}_\epsilon\cup\left\{\tilde S\subset X:S\in\mathcal{G}\mysetminus\mathcal{H}_\epsilon\right\}$ is still a $\delta$-cover of $B$, we have
	\begin{equation*}
		\phi_{\zeta,\delta}(B)\le \sum_{S\in \mathcal {H}_\epsilon}\zeta(S)+\sum_{S\in \mathcal {G}\mysetminus \mathcal {H}_\epsilon} \zeta(\tilde S)<t^{-1}\sum_{S\in \mathcal {H}_\epsilon}\mu(S)+\eta t^{-1}\epsilon\le t^{-1}[\mu(V)+\eta\epsilon].
	\end{equation*}
Since $\epsilon>0$ is independent of $\delta$, we let $\ep\to0^+$,
proving that $t\phi_{\zeta,\delta}(B)\le \mu(V)$. The arbitrary choice of $\delta$ concludes the proof.
\end{proof}

The previous lemmas imply the following characterization of the absolutely continuity
between two measures.

\bpr\label{remark abs cont}
Let $\mu$ be a measure over $X$, let $A\subset X$ be nonempty and 
consider a gauge $\zeta:\cS\to[0,+\infty]$.
We assume that the following conditions hold.
\begin{enumerate}
	\item $\mu$ is a regular measure.
	\item 
Every element of $\Sr$ is closed and $\mu$-measurable. 
\item
$\Sr$ covers $A$ finely. 
\item\label{musigmafinite}
$A$ has a countable covering whose elements are open and have $\mu$-finite measure.
\item
We have $c\ge1$ and $\eta>0$ such that for every $S\in \Sr$
there exists $\tilde S\in\cS$ such that
\begin{equation*}
	\hat S\subset\tilde S,\hspace{1cm}\diam(\tilde S)\le c\, \diam (S)\hspace{1cm}and\hspace{1cm}\zeta(\tilde S)\le\eta\zeta(S),
\end{equation*}
where $\hat S$ is the enlargement of $S$ defined in \eqref{d:tau-enlargement}.
\end{enumerate}
Then the condition $\mu{\RADot A}<<{\psir}{\RADot A}$  is equivalent to
\begin{equation}\label{un}
	\mu\left(\left\{x\in A:F^\zeta(\mu,x)=+\infty\right\}\right)=0.
\end{equation}
\epr
\begin{proof}
	We first suppose that $\mu{\RADot A}<<{\psir}{\RADot A}$. Let $\left\{V_k\right\}_{k\in\mathbb{N}}$ be a countable covering of $A$, whose elements are open and have $\mu$-finite measure. By Lemma \ref{lemma major},
	for every $t>0$ we have
	\begin{equation*}
		t\psir\left(\left\{x\in A:F^\zeta(\mu,x)=+\infty\right\}\cap V_k\right)
		\le\mu(V_k)<+\infty ,
	\end{equation*}
	so that $\psir\left(\{x\in A:F^\zeta(\mu,x)=+\infty\}\cap V_k\right)=0$ for every $k\in\mathbb{N}$, hence $$\psir\left(\left\{x\in A:F^\zeta(\mu,x)=+\infty\right\}\right)=0.$$ By our hypothesis, the equality (\ref{un}) immediately follows.
	
	Conversely, we assume that (\ref{un}) holds and consider $B\subset A$ with $\psir(B)=0$. We set 
	$$ B_{j}=\left\{x\in B:F^\zeta(\mu,x)<j\right\}$$
	and observe that 
	$B=\bigcup_{j\ge1}B_{j}\cup\left\{x\in B:F^\zeta(\mu,x)=+\infty\right\}$.
	By \eqref{un}, it follows that 
	$$\mu(B)\le\sum _{j\ge1}\mu(B_{j}).$$ Finally, we can apply Lemma~\ref{lemma minor} to obtain that  
	$$\mu(B_{j})\le j\psi_\zeta(B_{j})\le j\psi_\zeta(B)=0$$
	for all $j\ge1$, hence completing the proof.
\end{proof}

\begin{Rem}\rm
It is important to observe that in Proposition~\ref{remark abs cont} the absolute continuity between the two measures follows from \eqref{un} 
without assuming condition \eqref{musigmafinite}.
\end{Rem}	

\section{Measure-theoretic area formulas}

This section is devoted to the proof of the measure-theoretic area formula under
two different groups of assumptions. 
Combining Lemma~\ref{lemma minor}, Lemma~\ref{lemma major} and Proposition~\ref{remark abs cont}, we get a first version of the measure-theoretic area formula.
%
%
%
%
%
%
%
%
%
\begin{The}[Measure-theoretic area formula I]\label{the:meastheoarea}
	Let $\mu$ be a measure over $X$ and consider a gauge $\zeta:\cS\to[0,+\infty]$.
	We fix a nonempty set $A\subset X$ and assume that 
	the following conditions hold.
	\begin{enumerate}
		\item 
		$\mu$ is both a regular measure and a Borel measure.
		\item
		Every element of $\Sr$ is closed.
		\item
		$\cS_{\mu,\zeta}$ covers $A$ finely.
		\item
		$A$ is a Borel set.
		\item\label{coveringmufiniteA}
		$A$ has a countable covering whose elements are open and have $\mu$-finite measure.
		\item\label{ceta-estim}
		We have $c\ge1$ and $\eta>0$ such that for every $S\in \Sr$
		there exists $\tilde S\in\cS$ such that
		\begin{equation*}
			\hat S\subset\tilde S,\hspace{1cm}\diam(\tilde S)\le c\, \diam (S)\hspace{1cm}and\hspace{1cm}\zeta(\tilde S)\le\eta\zeta(S),
		\end{equation*}
		where $\hat S$ is the enlargement of $S$ defined in \eqref{d:tau-enlargement}.
		\item
		The subset $\lbrace x\in A: F^\zeta(\mu,x)=0\rbrace$ is
		$\sigma$-finite with respect to $\psi_\zeta$.
		\item
	We have the absolute continuity $\mu\res A<<\psi_\zeta\res A$.
	\end{enumerate}
	If $F^\zeta(\mu,\cdot):A\to[0,+\infty]$ is Borel, then for every Borel set $B\subset A$ we have
	\begin{equation}\label{eq:mestheoarea}
		\mu(B)=\int_B F^\zeta(\mu,x)\,d\psi_\zeta(x)\,.
	\end{equation}
\end{The}
\begin{proof}
Let $A^0=\{x\in A:F^\zeta(\mu,x)=0\}$ be covered by a countable
family $\{E_k\}_{k\in\N}$ of subsets such that $\psi_\zeta(E_k)<+\infty$ 
for every $k\in\N$. For $h>0$, applying Lemma~\ref{lemma minor}, we get
$$
\mu(A^0\cap E_k)\le h\psi_\zeta(A^0\cap E_k)\le h\psi_\zeta(E_k)<+\infty.
$$
Letting $h\to0^+$, we conclude that $\mu(A^0)=0$.
We choose a Borel set $B\subset A$ and an arbitrary $\lambda\in(0,1)$,
then we define
	$$B_{k}=\{x\in B:\lambda^{k+1}<F^\zeta(\mu,x)\le\lambda^k\}$$
for every $k\in\Z$, with $B_{+\infty}=\{x\in B:F^\zeta(\mu,x)=0\}$
	and $B_{-\infty}=\{x\in B:F^\zeta(\mu,x)=+\infty\}$.
These are all Borel sets. 
In particular, we have shown that 
replacing $B$ with $B_{+\infty}$ \eqref{eq:mestheoarea} holds.
Proposition~\ref{remark abs cont} shows that 
$\mu(\{x\in A:F^\zeta(\mu,x)=+\infty\})=0,$ therefore
$\mu(B_{-\infty})=0$.
Since $\mu$ is Borel and assumption (5) holds, for an arbitrary $h>0$
we can find an open set $W_h\supset B_{-\infty}$ such that
$\mu(W_h\sm B_{-\infty})=\mu(W_h)<h$. We apply Lemma~\ref{lemma major} to $B_{-\infty}$ with $t=1$, getting 
\[
h>\mu(W_h)>\psi_\zeta(B_{-\infty}).
\]
We conclude that $\psi_\zeta(B_{-\infty})=0$ and \eqref{eq:mestheoarea} 
also holds with $B_{-\infty}$ in place of $B$.
We are left to prove that \eqref{eq:mestheoarea} holds with 
$B$ replaced by $B'=B\sm(B_{+\infty}\cup B_{-\infty})$.
Using Lemma \ref{lemma minor}, for every $\epsilon>0$, we get
	\begin{equation*}
		\begin{split}
			\mu(B')=\sum_{k\in\mathbb{Z}}\mu(B_{k})&\le\sum_{k\in\mathbb{Z}}(1+\epsilon)\lambda^k\psi_\zeta(B_{k})\\
			&=(1+\epsilon)\lambda^{-1}\sum_{k\in\mathbb{Z}}\lambda^{k+1}\psi_\zeta(B_{k})\\
			&\le(1+\epsilon)\lambda^{-1}\int_{B'} F^\zeta(\mu,x)d\psi_\zeta(x).
		\end{split}
	\end{equation*}
Taking into account the arbitrary choice of $\epsilon$ and $\lambda$, we
have proved that
$$\mu(B')\le\int_{B'} F^\zeta(\mu,x)d\psi_\zeta(x).$$
To prove the opposite inequality, we consider the nontrivial case where
$\mu(B')<+\infty$. Taking into account that 
$\mu$ is Borel and the assumption \eqref{coveringmufiniteA} holds,
for every $\epsilon>0$ and every $k\in\mathbb{Z}$ there exists an open set $W_k\supseteq B_k$ such that $\mu(W_k\mysetminus B_k)<\epsilon2^{-|k|}$. Since $\mu(B_k)<+\infty$ for all $k\in\Z$, it follows that
	\begin{equation*}
		\begin{split}
			\mu(B')=&\sum_{k\in\mathbb{Z}}\mu(B_k)
			>\sum_{k\in\mathbb{Z}}\left(\mu(W_k)-\epsilon2^{-|k|}\right)\\
			=& \sum_{k\in\mathbb{Z}}\mu(W_k)-3\epsilon.
		\end{split}
	\end{equation*}
	Using Lemma~\ref{lemma major}, we get
	\begin{equation*}
\sum_{k\in\mathbb{Z}}\mu(W_k)\ge\sum_{k\in\mathbb{Z}}\lambda^{k+1}\psi_\zeta(B_k)\ge\lambda\int_{B'} F^\zeta(\mu,x)d\psi_\zeta(x).
	\end{equation*}
	We have proved that  
	$$\mu(B')\ge\lambda \int_{B'} F^\zeta(\mu,x)d\psi_\zeta(x)-3\epsilon.$$
	The arbitrary choice of $\epsilon$ and $\lambda$ 
	lead us to the opposite inequality
	$$\mu(B')\ge\int_{B'} F^\zeta(\mu,x)d\psi_\zeta(x),$$
	hence concluding the proof.
\end{proof}
\begin{Rem}\rm 
We notice that the use of Proposition~\ref{remark abs cont} 
in the proof of Theorem~\ref{the:meastheoarea} depends on our
choice to present the assumptions. In fact, if we replace the absolute continuity $\mu{\RADot A}<<{\psir}{\RADot A}$ with the condition 
\begin{equation}\label{eq:thetalessinfI}
		\mu\ls\lbrace x\in A: F^\zeta(\mu,x)=+\infty\rbrace\rs=0,
\end{equation}
then Proposition~\ref{remark abs cont} is not necessary in the proof.
\end{Rem}

In the next version of the measure-theoretic area formula, the Borel measurability of
the Federer density is replaced by the more general $\psir$-measurability.
We can also consider a more general $\psir$-measurable domain of integration $A$,
requiring in addition the $\sigma$-finiteness with respect to $\psir$.
%
%
%
%
%
%
%
\begin{The}[Measure-theoretic area formula II]\label{theo-mea area form}
Let $\mu$ be a measure over $X$ and consider a gauge $\zeta:\cS\to[0,+\infty]$.
We fix a nonempty set $A\subset X$ and assume that 
the following conditions hold.
\begin{enumerate}
	\item 
	$\mu$ is a Borel regular measure.
	\item
	Every element of $\cS$ is Borel and every element of $\Sr$ is closed.
	\item
	$\Sr$ covers $A$ finely.
	\item
	$A$ is $\psi_\zeta$-measurable and $\sigma$-finite with respect to $\psir$.
	\item\label{eq:countablefinitemu}
	$A$ has a countable covering whose elements are open and have $\mu$-finite measure.
	\item\label{ceta-estim-II}
		We have $c\ge1$ and $\eta>0$ such that for every $S\in \Sr$
		there exists $\tilde S\in\cS$ such that
		\begin{equation*}
			\hat S\subset\tilde S,\hspace{1cm}\diam(\tilde S)\le c\, \diam (S)\hspace{1cm}and\hspace{1cm}\zeta(\tilde S)\le\eta\zeta(S),
		\end{equation*}
		where $\hat S$ is the enlargement of $S$ defined in \eqref{d:tau-enlargement}.
	\item
	We have the absolute continuity $\mu\res A<<\psi_\zeta\res A$.
\end{enumerate}
If $F^\zeta(\mu,\cdot):A\to[0,+\infty]$ is $\psi_\zeta$-measurable, then every $\psi_\zeta$-measurable subset $B\subset A$
	is also $\mu$-measurable and the following formula holds
	\begin{equation}\label{eq:mestheoareaI}
		\mu(B)=\int_B F^\zeta(\mu,x)\,d\psi_\zeta(x)\,.
	\end{equation}
\end{The}
\begin{proof}
	Let $B\subset A$ be a $\psir$-measurable set.
	It is not restrictive to assume that $\psir(B)<+\infty$. 
	This is justified by the fact that $A$ can be written as a disjoint union $\bigcup_{i\in\mathbb{N}} A_k$, where $A_k$ is $\psir$-measurable and it has $\psir$-finite measure.
Condition (2) implies that $\psir$ is Borel regular, so even if we may have a priori a sequence of arbitrary sets that cover $A$ and have $\psir$-finite measure,
we can always assume they are Borel and in particular $\psir$-measurable.
Our claim for every $B\cap A_k$ implies that this set is $\mu$-measurable for every $k\in\N$ and 
	\begin{equation*}
	\mu(B)=\sum_{k\in\mathbb{N}}\mu(B\cap A_k)=\sum_{k\in\mathbb{N}}\int_{B\cap A_k} F^\zeta(\mu,x)d\psir(x)=\int_B F^\zeta(\mu,x)d\psir(x).
	\end{equation*}
The assumption that $\psi_\zeta(B)<+\infty$ in particular gives
	\begin{equation}\label{uf}
	\psir\left(\left\{x\in B:F^\zeta(\mu,x)=0\right\}\right)<+\infty.
\end{equation} 
Since $\psir$ is Borel regular and $\psir(B)<+\infty$, we can write $B=B_0\cup N\subset A$, where
	$B_0$ is Borel and $\psir(N)=0$. By our assumption, $\mu(N)=0$, hence $B$ is also $\mu$-measurable. Since $A$ is $\sigma$-finite with respect to $\psir$, actually any $\psir$-measurable subset of $A$ is also $\mu$-measurable. 
	
	We fix $0<\lambda<1$ and for every $k\in\mathbb{Z}$ define
	$$B_{k}=\{x\in B:\lambda^{k}<F^\zeta(\mu,x)\le\lambda^{k+1}\},$$
that is $\psir$-measurable, along with
\[
B_{+\infty}=\{x\in B:F^\zeta(\mu,x)=0\}\quad \text{and}\quad 
B_{-\infty}=\{x\in B:F^\zeta(\mu,x)=+\infty\}.
\] 
Thus, all of these sets are also $\mu$-measurable.
Since $\psir(B_{+\infty})<+\infty$, for any fixed $\kappa>0$, due to Lemma~\ref{lemma minor}, we get
$$
\mu(B_{+\infty})\le \kappa \psi_\zeta(B_{+\infty})<+\infty.
$$
The arbitrary choice of $\kappa>0$ gives $\mu(B_{+\infty})=0$.
This proves \eqref{eq:mestheoareaI} for $B=B_{+\infty}$.
In view of Proposition~\ref{remark abs cont}, we get
	\begin{equation*}
		\mu\left(\left\{x\in A:F^\zeta(\mu,x)=+\infty\right\}\right)=0
	\end{equation*}
and in particular $\mu(B_{-\infty})=0$.
We now use the fact that $\mu$ is Borel and assumption (5) holds.
Then for any $h>0$ there exists an open set $W_h\supset B_{-\infty}$ 
such that
$\mu(W_h\sm B_{-\infty})=\mu(W_h)<h$.
By Lemma~\ref{lemma major} with $t=1$, we get 
$h>\mu(W_h)>\psi_\zeta(B_{-\infty})$,
hence $\psi_\zeta(B_{-\infty})=0$ and \eqref{eq:mestheoarea} 
holds with $B=B_{-\infty}$.
We set $B'=B\sm(B_{+\infty}\cup B_{-\infty})$ and observe that
the validity of \eqref{eq:mestheoareaI} for $B=B'$ concludes the proof.
Due to Lemma~\ref{lemma minor}, for $\epsilon>0$
it holds  
	\begin{equation*}
		\begin{split}
			\mu(B')=\sum_{k\in\mathbb{Z}}\mu(B_{k})&\le
			\sum_{k\in\mathbb{Z}}(1+\epsilon)
			\lambda^{k}\psi_\zeta(B_{k})\\
			&\le(1+\epsilon)\lambda^{-1}\sum_{k\in\mathbb{Z}}\lambda^{k+1}\psi_\zeta(B_{k})\\
			&\le(1+\epsilon)\lambda^{-1}\int_{B'} F^\zeta(\mu,x)d\psi_\zeta(x).
		\end{split}
	\end{equation*}
The arbitrary choice of $\epsilon$ and $\lambda$, give the first inequality
	$$\mu(B')\le\int_{B'} F^\zeta(\mu,x)d\psi_\zeta(x).$$
We now prove the opposite inequality, considering the nontrivial case
where $\mu(B')<+\infty$. As before, the Borel regularity of $\mu$ and assumption (5) provide 
an outer approximation with open sets. 
For every $\epsilon>0$ and $k\in\mathbb{Z}$ there exists an open set 
$W_k\supseteq B_k$ such that 
	$$
	\mu(W_k\mysetminus B_k)<\epsilon2^{-|k|}.
	$$
We obtain the following inequalities
\begin{equation*}
\mu(B')=\sum_{k\in\mathbb{Z}}\mu(B_k)
			>\sum_{k\in\mathbb{Z}}\left(\mu(W_k)-\epsilon2^{-|k|}\right)\\
			=\sum_{k\in\mathbb{Z}}\mu(W_k)-3\epsilon.
\end{equation*}
Our assumptions also allow to apply Lemma \ref{lemma major}, hence
\begin{equation*}
\sum_{k\in\mathbb{Z}}\mu(W_k)\ge\sum_{k\in\mathbb{Z}}\lambda^{k+1}\psi_\zeta(B_k)	\ge\lambda\int_{B'} F^\zeta(\mu,x)d\psi_\zeta(x).
\end{equation*}
It follows that  
$$\mu(B')\ge\lambda\int_{B'} F^\zeta(\mu,x)d\psi_\zeta(x)-3\epsilon.$$
The arbitrary choice of $\ep>0$ and $\lambda$ provide 
	the opposite 	inequality
	$$\mu(B')\ge\int_{B'} F^\zeta(\mu,x)d\psi_\zeta(x),$$
	hence completing the proof.
\end{proof}

\begin{Rem}\rm
We notice that in both Theorem~\ref{the:meastheoarea} and Theorem~\ref{theo-mea area form},
assuming that $\mu$ is finite on all open metric balls gives the assumption \eqref{eq:countablefinitemu}.
\end{Rem}

\section{Applications to Hausdorff measures}\label{sect:HausdSpher}

The special cases where $\psir$ is either the Hausdorff measure or the spherical measure 
are important for the applications in Geometric Measure Theory.
The next theorems consider the case where the gauge is proportional to some power of the diameter. 
We prove that some of the ``abstract conditions" for measure-theoretic area formulas are satisfied. 
%
%
%
%
%
%
%
%
\begin{The}\label{th:Hmeasurability}
Let $\mu$ be a measure over $X$, $\alpha>0$ and fix $A\subset X$.
If $\cF_{\mu,\zeta_\alpha}$ covers $A$ finely, then $\cd^\alpha(\mu,\cdot):A\to[0,+\infty]$ is Borel with respect to the subspace topology of $A$. 
\end{The}
\begin{proof}
We have to prove that $\cd^\alpha(\mu,\cdot)$ is Borel with respect to the subspace topology of $A$. It suffices to prove that for any $\delta>0$ the function 
$\cd^\alpha_\delta(\mu,\cdot):A\to [0,+\infty]$ defined as
	$$
	\cd^\alpha_\delta(\mu,x)=\sup\{Q_{\mu,\zeta_\alpha}(S):x\in S,\, S\in\cF_{\mu,\zeta_\alpha},\diam(S)<\delta\}
	$$
is Borel with respect to the subspace topology of $A$.
We fix $t>0$ and consider the set 
$$\cS_{t,\delta}=\{x\in A:\,\cd^\alpha_\delta(\mu,x) >t\}.$$
We wish to prove that $\cS_{t,\delta}$ is open in $A$. 
Let us fix $y\in\cS_{t,\delta}$. There exists a set $S\in\cF_{\mu,\zeta_\alpha}$ 
containing $y$, such that
$$
Q_{\mu,\zeta_\alpha}(S)>t\quad\text{and}\quad \diam(S)<\delta.
$$ 
If $\zeta_\alpha(S)=0$, then $S=\{y\}$ and $\mu(\{y\})>0$. 
We have only two possibilities: either there exists $\epsilon>0$ such that $\text{diam}(\mathbb{B}(y,\epsilon))=0$ or we can find 
\[
0<\epsilon<\min\left\{\frac\delta2,\frac 12\left(\frac{\mu(\set{y})}{tc_\alpha}\right)^{1/\alpha}  \right\}
\]
such that $\diam(\B(y,\epsilon))>0$. 
In the former case we get $Q_{\mu,\zeta_\alpha}(\B(y,\ep))=+\infty$ and we can also assume that $\ep<\delta/2$. In the latter, we obtain 
$$
Q_{\mu,\zeta_\alpha}(\B(y,\epsilon))=\frac{\mu(\mathbb{B}(y,\epsilon))}{c_\alpha \text{diam}(\mathbb{B}(y,\epsilon))^\alpha}\ge\frac{\mu(\{y\})}{c_\alpha( 2\epsilon)^\alpha}>t.$$
In both cases $\B(y,\ep)\in\cF_{\mu,\zeta_\alpha}$, $\diam(\B(y,\ep))<\delta$ and  
$$Q_{\mu,\zeta_\alpha}(\mathbb{B}(y,\epsilon))>t.$$
If we choose any $z\in\mathbb{B}(y,\epsilon)$, then the definition of Federer density gives
$$
t<Q_{\mu,\zeta^\alpha}(\mathbb{B}(y,\epsilon))\le\cd^\alpha_\delta(\mu,z).
$$
We have shown that ${B}(y,\epsilon)\cap A\subset \cS_{t,\delta}$.
Now we consider the case $\zeta^\alpha(S)>0$, where 
$$
Q_{\mu,\zeta^\alpha}(S)=\frac{\mu(S)}{\zeta^\alpha(S)}>t
$$ 
and $\diam(S)<\delta$. By the triangle inequality, since $y\in S$, there exists an $\epsilon>0$ such that $$\diam(S\cup\mathbb{B}(y,\epsilon))\le \text{diam}(S)+2\epsilon<\delta.$$ Choosing a possibly smaller $\epsilon>0$, we also get 
$$
\frac{\mu(S)}{c_\alpha\diam(S\cup\mathbb{B}(y,\epsilon))^\alpha}>t,
$$ where 
$S\cup\mathbb{B}(y,\epsilon)$ is still closed and belongs to $\cF_{\mu,\zeta_\alpha}$.
We have proved that $$\frac{\mu(S\cup\mathbb{B}(y,\epsilon))}{\zeta^\alpha(S\cup\mathbb{B}(y,\epsilon))}\ge\frac{\mu(S)}{c_\alpha\diam(S\cup\mathbb{B}(y,\epsilon))^\alpha}>t.$$
For every $z\in{B}(y,\epsilon)\cap A$, we obtain 
$$
t<Q_{\mu,\zeta^\alpha}(S\cup\mathbb{B}(y,\epsilon))\le \cd^\alpha_\delta(\mu,z)
$$
hence 
${B}(y,\epsilon)\cap A\subset \cS_{t,\delta}$.
Thus there exists an open set $\Omega$ such that $\cS_{t,\delta}=\Omega\cap A$
and we may conclude that $\cd^\alpha_\delta(\mu,\cdot)$ is lower semicontinuous on $A$ with
respect to the subspace topology of $A$. Thus, $y\to \cd^\alpha(\mu,y)$ is Borel with respect to the subspace topology of $A$. 
\end{proof}
\begin{Rem}\label{rem:F'}\rm 
	We denote by $\cF''$ the family of closed bounded sets with positive diameter and observe that $\cF''=(\cF'')_{\mu,\zeta_\alpha}$. Then the assumption that $\cF''$ covers $A$ finely implies that $(\cF)_{\mu,\zeta_\alpha}$ covers $A$ finely and from Theorem~\ref{th:Hmeasurability} we can conclude that
$\cd^\alpha(\mu,\cdot)$ is well defined and Borel with respect to the subspace topology of $A$.
    If for every $x\in A$ there exists a sequence of closed bounded sets $F_{x,k}$ (depending on $x$)
    that contain $x$ and have positive diameter for every $k\in\N$, with $\diam(F_{x,k})\to0$, then $\cF''$ covers $A$ finely and so does $(\cF)_{\mu,\zeta_\alpha}$
\end{Rem}
%
%
%
%
%
%
%
%
%
\begin{The}[Area formula for the Hausdorff measure I]\label{area formula haus Borel}
	Let $\mu$ be a measure over $X$, $\alpha>0$ and fix $A\subset X$.
	We assume the validity of the following conditions.
	\begin{enumerate}
		\item 
		$\mu$ is both a regular measure and a Borel measure.
		\item
		$\cF_{\mu,\zeta_\alpha}$ covers $A$ finely.
		\item  
		$A$ is a Borel set.
		\item
		$A$ has a countable covering whose elements are open and have $\mu$-finite measure.
		\item
		The subset $\lbrace x\in A: \cd^\alpha(\mu,x)=0\rbrace$ is
		$\sigma$-finite with respect to $\cH^\alpha$.
    	\item
		We have the absolute continuity $\mu\res A<<\mathcal{H}^\alpha\res A$.
	\end{enumerate}
	Then $\cd^\alpha(\mu,\cdot):A\to[0,+\infty]$ is Borel and for every Borel set $B\subset A$ we have 
	\begin{equation}
	\mu(B)=\int_B\cd^\alpha(\mu,x)\,d\mathcal{H}^\alpha(x).
	\end{equation}
\end{The}
\begin{proof}
By Theorem~\ref{th:Hmeasurability} the function $\cd^\alpha(\mu,\cdot):A\to[0,+\infty]$ is Borel with respect to the subspace topology of $A$, hence it is also Borel. Condition \eqref{ceta-estim} of Theorem~\ref{the:meastheoarea} is satisfied by taking $\tilde S$ equal to the topological closure of $\hat S$, 
defined in \eqref{d:tau-enlargement}. Thus, we are in the assumptions of Theorem~\ref{the:meastheoarea} and our claim is established.
\end{proof}	
%
%
%
%
%
%
%
%
%
%
\begin{The}[Area formula for the Hausdorff measure II]\label{area formula haus}
Let $\mu$ be a measure over $X$, $\alpha>0$ and fix $A\subset X$.
We assume the validity of the following conditions.
	\begin{enumerate}
		\item 
		$\mu$ is a Borel regular measure.
		\item
		$\cF_{\mu,\zeta_\alpha}$ covers $A$ finely.
		\item  
		$A$ is ${\mathcal H^\alpha}$-measurable and $\sigma$-finite with respect to ${\mathcal{H}^\alpha}$.
		\item
		$A$ has a countable covering whose elements are open and have $\mu$-finite measure.
		\item
		We have the absolute continuity $\mu\res A<<\mathcal{H}^\alpha\res A$.
	\end{enumerate}
	Then  $\cd^\alpha(\mu,\cdot):A\to[0,+\infty]$ is $\cH^\alpha$-measurable,
	every $\mathcal{H}^\alpha$-measurable set $B\subset A$ is also $\mu$-measurable and
	we have 
	\begin{equation}
		\mu(B)=\int_B\cd^\alpha(\mu,x)\,d\mathcal{H}^\alpha(x).
	\end{equation}
\end{The}
\begin{proof}
Theorem~\ref{th:Hmeasurability} shows that $\cd^\alpha(\mu,\cdot):A\to[0,+\infty]$ is Borel with respect to the subspace topology of $A$. Since $A$ is $\cH^\alpha$-measurable,
we immediately observe that  $\cd^\alpha(\mu,\cdot)$ is also $\cH^\alpha$-measurable.
Condition \eqref{ceta-estim-II} of Theorem~\ref{theo-mea area form} is satisfied by taking $\tilde S$ equal to the topological closure of $\hat S$, 
	defined in \eqref{d:tau-enlargement}. Thus, the application of Theorem~\ref{theo-mea area form} concludes the proof.
\end{proof}

\begin{Rem}\rm 
If $\cF''$ is as in Remark~\ref{rem:F'}, then in the assumptions of either Theorem~\ref{area formula haus Borel} or Theorem~\ref{area formula haus} the collection $\cF''$ covers $A$ finely if and only if so does $(\cF)_{\mu,\zeta_{\alpha}}$.
As noticed in Remark~\ref{rem:F'}, from the equality $\cF''=(\cF'')_{\mu,\zeta_{\alpha}}\subset (\cF)_{\mu,\zeta_{\alpha}}$, if $\cF''$ covers $A$ finely clearly so does $(\cF)_{\mu,\zeta_\alpha}$. Conversely, if $(\cF)_{\mu,\zeta_{\alpha}}$ covers $A$ finely, then we fix any $x\in A$, getting a sequence of closed sets $\set{F_{x,k}}\subset(\cF)_{\mu,\zeta_{\alpha}}$ that all contain $x$ and
such that $\diam(F_{x,k})\to0$.
The absolute continuity of $\mu\res A$ with respect to $\cS^\alpha\res A$
implies that $(\cF)_{\mu,\zeta_{\alpha}}$ contains no points.
As a result, for all $k$ sufficiently large, we have $F_{x,k}\in\cF''$.
From the arbitrary choice of $x$ we infer that $\cF''$ covers $A$ finely.
\end{Rem}

The next area formulas will replace the Hausdorff measure with the spherical measure.
This requires a minimal regularity condition on the diameter function.
\begin{Def}[Diametric regularity]\label{diam reg}\rm
	A metric space $X$ is called \textit{diametrically regular} if for every $x\in X$ there exist an $R_x>0$ and a $\delta_{x}>0$ such that $(0,\delta_x)\ni r\to \text{diam}(B(y,r))$ is a continuous function for every $y\in{B}(x,R_x)$.
\end{Def}
\begin{Lem}\label{lemma sph}
If a metric space $X$ is diametrically regular, then for every $x\in X$ there exist $R_x>0$ and $\delta_x>0$ such that whenever $0<t<\delta_x$ and $y\in{B}(x,R_x)$, we have
	$$\diam(B(y,t))= \diam(\mathbb B(y,t)).$$
\end{Lem}
\begin{proof}
Let $x\in X$ and choose $R_x$ and $\delta_x$ as in Definition~\ref{diam reg}.	
For $0<t<\delta_x$, we consider any $t'\in(t,\delta_x)$ and observe that
$\diam(\mathbb B(y,t))\le \diam (B(y,t'))$ for all $y\in B(x,R_x)$. By diametric regularity,
passing to the limit as $t'\searrow t$, we obtain $\diam(\mathbb B(y,t))\le \diam (B(y,t))$,
that concludes the proof.
\end{proof}

\begin{Lem}\label{lem:coversx}
Let $\zeta_{o,\alpha}:\cF_o\to[0,+\infty)$ be defined on the family $\cF_o$ of all open balls of $X$ as $\zeta_{o,\alpha}(B)=c_\alpha\diam(B)^\alpha$ for every $B\in\cF_o$. 
We fix $x\in X$, $\alpha>0$ and consider a measure $\mu$ over $X$.
If $(\cF_b)_{\mu,\zeta_{b,\alpha}}$ covers $\set{x}$ finely, then
so does $(\cF_o)_{\mu,\zeta_{o,\alpha}}$.
\end{Lem}
\begin{proof}
We consider an infinitesimal sequence of positive numbers $\epsilon_k\searrow 0$. By assumption, there exists
a sequence of closed metric balls $\B(x_k,r_k)\in(\cF_b)_{\mu,\zeta_{b,\alpha}}$ 
such that $x\in \B(x_k,r_k)$ and $\diam(\B(x_k,r_k))< \ep_k$ for all $k\in\N$.
We have only two possible cases. In the first one there exists $k_0\in\N$ such that
$\diam(\B(x_k,r_k))=0$ for all $k\ge k_0$.
For these $k$'s, we get $\B(x_k,r_k)=\set{x}$, hence $\mu(\set{x})>0$. 
If $\diam(B(x,\ep_k/4))>0$, then clearly $B(x,\ep_k/4)\in(\cF_o)_{\mu,\zeta_{o,\alpha}}$
and $\diam(B(x,\ep_k/4))<\ep_k$. 
In the remaining case, there exists a subsequence $r_{\alpha(k)}$ of $r_k$ such that
$0<\diam(\B(x_{\alpha(k)},r_{\alpha(k)}))<\ep_{\alpha(k)}$ and $\B(x_{\alpha(k)},r_{\alpha(k)})\in(\cF_b)_{\mu,\zeta_{b,\alpha}}$.
We now fix any $k\in\N$. There exists $k_1>k$ such that 
$$0<2\,\diam(\B(x_{\alpha(k_1)},r_{\alpha(k_1)}))<\diam(\B(x_{\alpha(k)},r_{\alpha(k)}))<\ep_{\alpha(k)}$$
and $x\in\B(x_{\alpha(k_1)},r_{\alpha(k_1)})$.
We observe that for every $y\in \B(x_{\alpha(k_1)},r_{\alpha(k_1)})$, we have
\[
d(y,x)\le d(y,x_{\alpha(k_1)})+d(x,x_{\alpha(k_1)})\le 2\,\diam(\B(x_{\alpha(k_1)},r_{\alpha(k_1)}))<\ep_{\alpha(k)},
\]
therefore $\B(x_{\alpha(k_1)},r_{\alpha(k_1)})\subset B(x,\ep_{\alpha(k)})$.
We have proved that $\diam(B(x,\ep_{\alpha(k)}))>0$ for every $k\in\N$,
hence $B(x,\ep_{\alpha(k)})\in(\cF_o)_{\mu,\zeta_{o,\alpha}}$.
In any of the two cases $(\cF_o)_{\mu,\zeta_{o,\alpha}}$ covers $\set{x}$ finely.
\end{proof}

\begin{The}\label{th:spher Borel}
Let $X$ be diametrically regular, $\alpha>0$ and let $\mu$ be a measure over $X$. 
If $A\subset X$ and $(\cF_b)_{\mu,\zeta_{b,\alpha}}$ covers $A$ finely, then $\cs^\alpha(\mu,\cdot):A\to[0,+\infty]$ is Borel with respect to the subspace topology of $A$.
\end{The}

\begin{proof}
We consider the gauge $\zeta_{o,\alpha}$ of Lemma~\ref{lem:coversx}.
The lemma shows that in our assumptions also 
$(\cF_o)_{\mu,\zeta_{o,\alpha}}$ covers $A$ finely, so also the
Federer density $F^{\zeta_{o,\alpha}}(\mu,x)$ is well defined at each point $x\in A$.
Our first claim is the following equality
\begin{equation}\label{eq:salphaFalpha}
	\cs^\alpha(\mu,x)=F^{\zeta_{o,\alpha}}(\mu,x)
\end{equation}
for every $x\in A$. We fix $x\in A$ and consider $R_x>0$ and $\delta_x>0$ 
of Definition~\ref{diam reg}.
We wish to prove that 
\[
F^{\zeta_{o,\alpha}}(\mu,x)\le\cs^\alpha(\mu,x).
\]
It is equivalent to showing that
	\begin{equation}\label{eq 2.11'}
		\inf_{\delta>0}\sup_{\begin{smallmatrix}&x\in B\in(\cF_o)_{\mu,\zeta_{o,\alpha}}\\ &\text{diam}(B)<\delta\end{smallmatrix}}Q_{\mu,\zeta_{o,\alpha}}(B)\le
		\inf_{\delta>0}\sup_{\begin{smallmatrix}&x\in B\in(\cF_b)_{\mu,\zeta_{b,\alpha}}\\ &\text{diam}(B)<\delta\end{smallmatrix}}Q_{\mu,\zeta_{b,\alpha}}(B).
	\end{equation}
We fix $\delta>0$ such that $\delta<\min\{R_x,\delta_x\}$ and choose
$B\in(\cF_o)_{\mu,\zeta_{o,\alpha}}$ such that $x\in B$ and $\diam(B)<\delta$. We may write $B=B(y,r)$ for some $y\in X$ and $r>0$. We define $$t_y=\sup\{d(y,z):z\in B\}\le r$$
and notice that clearly $B\subset \mathbb B(y,t_y)$, hence \begin{equation}\label{2.100}
	\mu(B)\le\mu(\mathbb B(y,t_y)).
\end{equation}
We observe that $t_y\le r$ and $$t_y\le\diam(B)<\delta<\min\{R_x,\delta_x\}.$$
In particular, $y\in B(x,R_x)$ and $t_y<\delta_x$.
Since $B\subset \mathbb B(y,t_y)\subset B(y,t_y')$ for every $t_y'>t_y$, then $\diam(B)\le \diam (B(y,t_y'))$ and by diametric regularity for 
$t_y'\searrow t_y$ and $t_y'<\delta_x$, we obtain
	\begin{equation}\label{2.200}
		\diam(B)\le\diam(B(y,t_y)).
	\end{equation}
	Moreover $B(y,t_y)\subset B(y,r)$ implies that 
	\begin{equation}\label{2.300}
		\diam(B(y,t_y))\le\diam(B(y,r)).
	\end{equation}
	By inequalities (\ref{2.200}), (\ref{2.300}) and by Lemma~\ref{lemma sph} applied to the ball ${B}(y,t_y)$, we obtain
	\begin{equation}\label{2.400}
		\diam(B)=\diam(\mathbb{B}(y,t_y)).
	\end{equation}
	If $\diam (B)>0$, by equation (\ref{2.400}) we have $\B(y,t_y)\in (\cF_b)_{\mu,\zeta_{b,\alpha}}$ and using (\ref{2.100}) and again equation (\ref{2.400}), we obtain
	$$Q_{\mu,\zeta_{o,\alpha}}(B)=\frac{\mu(B)}{c_\alpha\diam(B)^\alpha}\le\frac{\mu(\mathbb B(y,t_y))}{c_\alpha\diam(\mathbb B(y,t_y))^\alpha}=Q_{\mu,\zeta_{b,\alpha}}(\mathbb B(y,t_y)).$$
	If $\diam (B(y,r))=0$, taking into account that $B\in(\cF_o)_{\mu,\zeta_{o,\alpha}}$,
	we have $\mu(B(y,r))>0$. In addition, $B(y,r)=\mathbb B(y,\frac{r}{2})=\{y\}=\{x\}$ and $\diam(\mathbb B(x,\frac{r}{2}))=0$, hence
	$$Q_{\mu,\zeta_{o,\alpha}}\left(B(y,r)\right)=Q_{\mu,\zeta_{b,\alpha}}\left(\mathbb B\left(y,\frac{r}{2}\right)\right)=+\infty.$$
	We have proved that for every $\delta>0$ such that $\delta<\min\{R_x,\delta_x\}$
	the inequality (\ref{eq 2.11'}) holds. 
	
To prove the opposite inequality, we first consider the case $\cs^\alpha(\mu,x)=+\infty$.
Then we can find a family of closed balls $\{C_k\}_{k\in\N}\subset
(\cF_b)_{\mu,\zeta_{b,\alpha}}$ such that, for every $k\in\mathbb N$ we have $\diam(C_k)<2^{-k}$, $x\in C_k$ and 
$$
\lim_{k\to +\infty} Q_{\mu,\zeta_{b,\alpha}}(C_k)=+\infty.
$$
If there exists $k_0\in\N$ such that for every $k\in\mathbb{N}$ 
with $k\ge k_0$ we have $\diam(C_k)>0$, then there exists $\tilde r_k>0$ such
that $C_k=\B(x_k,\tilde r_k)$ and the following supremum satisfies
\[
r_k=\sup\left\{d(z,x_k):z\in C_k\right\}>0\quad \text{and}\quad r_k\le \tilde r_k.
\]
We have seen that $\B(x_k,r_k)\subset C_k$.
On the other hand, by definition of $r_k$ we also get 
\[
C_k\subset\mathbb B(x_k, r_k),
\]
hence $C_k=\mathbb B(x_k, r_k)$. 
We select any $r'_k\in(r_k,2 r_k]$ and observe that
$$
0<\diam(C_k)\le\diam(B(x_k,r'_k))\le2 r'_k\le4r_k\le4\diam(C_k),
$$
where the last inequality follows from the definition of $r_k$.
It follows that 
$$
\frac{\mu({B}(x_k,r'_k))}{\zeta_{o,\alpha}({B}(x_k,r'_k))}\ge
4^{-\alpha}\frac{\mu(C_k)}{\zeta_{b,\alpha}( C_k)}=4^{-\alpha}Q_{\mu,\zeta_{b,\alpha}}(C_k)\to+\infty
$$ 
as $k\to+\infty$. This esablishes the opposite inequality when the diameters of $C_k$ are positive for all large $k$'s. In the remaining case there exists a strictly increasing sequence of integers $\alpha(k)\in\N$ such that 
such that $\diam(C_{\alpha(k)})=0$ for all $k\in\N$.
It follows that $C_{\alpha(k)}=\{x\}$ and $\mu(\{x\})>0$. 
We choose any deacreasing sequence $s_k\searrow 0$, hence 
$$
\limsup_{k\to +\infty}
\frac{\mu({B}(x,s_k))}{\zeta^{o,\alpha}({B}(x,s_k))}\ge\limsup_{k\to +\infty}\frac{\mu(\{x\})}{c_\alpha\diam({B}(x,s_k))^\alpha}\ge\limsup_{k\to +\infty}\frac{\mu(\{x\})}{c_\alpha(2s_k)^\alpha}=+\infty.
$$
We have shown that $F^{\zeta_{o,\alpha}}(\mu,x)=+\infty$, hence concluding the
proof of the inequality
\begin{equation}\label{eq:FSalpha}
\cs^\alpha(\mu,x)\le F^{\zeta_{o,\alpha}}(\mu,x)
\end{equation}
in the case $\cs^\alpha(\mu,x)=+\infty$.  
We will prove \eqref{eq:FSalpha} when $\cs^\alpha(\mu,x)<+\infty$.
It is convenient to restrict our attention to the following subset of integers 
\begin{equation*}
		\Lambda=\{k\in\mathbb{N}:2^{-k}<\min\{R_x,\delta_{x}\}\},
\end{equation*}
where $R_x$ and $\delta_x$ are those of Definition~\ref{diam reg}.
By our assumptions, there exists a family 
$$\{D_k: k\in\Lambda\}\subset(\cF_b)_{\mu,\zeta_{b,\alpha}}$$
of closed balls such that $x\in D_k$ for all $k\in\N$ and there exists 
$k_1\in\N$ such that for every $k\in\Lambda$ with $k\ge k_1$ we have
$\diam(D_k)<2^{-k}$ and $$\sup\left\{Q_{\mu,\zeta^{b,\alpha}}(D):x\in D,\, D\in(\cF_b)_{\mu,\zeta_{b,\alpha}},\,\diam(D)<2^{-k}\right\}-2^{-k}<Q_{\mu,\zeta_{b,\alpha}}(D_k)<+\infty.
$$ As a consequence, for these $k$'s we have that $\zeta_{b,\alpha}(D_k)>0$ and $$Q_{\mu,\zeta_{b,\alpha}}(D_k)=\frac{\mu(D_k)}{\zeta_{b,\alpha}(D_k)}.$$
We choose any $k\in\Lambda$ with $k\ge k_0$.
Arguing as before, we find a center $y_k\in D_k$ and the following positive radius
\[
t_k=\sup\left\{d(z,y_k):z\in D_k\right\}\le \diam(D_k)<2^{-k}<\delta_x
\]
such that $D_k=\B(y_k,t_k)$. We know that
$d(y_k,x)\le \diam(D_k)<2^{-k}<R_x$,
therefore the diametric regularity and Lemma~\ref{lemma sph} imply that
$s\to\diam(B(y_k,s))=\diam(\B(y_k,s))$ is continuous on $(0,\delta_x)$.
Then we can find $t'_k\in(t_k,\delta_x)$ such that $0<\diam(B(y_k,t_k'))<2^{-k}$ and
\begin{equation}\label{2.12}
\frac{\mu(D_k)}{\zeta_{b,\alpha}(D_k)}-2^{-k}
<\frac{\mu(D_k)}{\zeta_{o,\alpha}(B(y_k,t_k'))}\le
		\frac{\mu(B(y_k,t_k'))}{\zeta_{o,\alpha}(B(y_k,t_k'))}.
\end{equation}
We have proved that 
$$
Q_{\mu,\zeta_{b,\alpha}}(D_k)-2^{-k}\le\sup
\left\{Q_{\mu,\zeta^{b,\alpha}}(B):\,x\in B, \, B\in(\cF_o)_{\mu,\zeta_{o,\alpha}},\diam(B)<2^{-k}\right\},
$$ 
for all $k\in\Lambda$ with $k\ge k_1$. As $k\to+\infty$, the opposite 
\eqref{eq:FSalpha} follows and the claim \eqref{eq:salphaFalpha} is established.
Thus, the proof is finished if we show that $F^{\zeta_{o,\alpha}}:A\to[0,+\infty]$ is 
Borel with respect to the subspace topology of $A$.
It suffices to show such measurability for 
$F^{\zeta_{o,\alpha}}_\delta(\mu,\cdot) :A\to [0,+\infty]$ defined as follows 
$$
F^{\zeta_{o,\alpha}}_\delta(\mu,x)
=\sup\left\{Q_{\mu,\zeta_{o,\alpha}}(B):\,x\in B,\, B\in (\cF_o)_{\mu,\zeta_{o,\alpha}},\text{diam(S)}<\delta\right\}$$
when $\delta>0$ is arbitrarily chosen.
We fix $t>0$ and consider the set 
$$
\cS_{t,\delta}=\{z\in A:\,{F^{\zeta_{o,\alpha}}_\delta}(\mu,z)>t\}.
$$ 
If $y\in\cS_{t,\delta}$, then we can find an open ball $B\in (\cF_o)_{\mu,\zeta_{o,\alpha}}$ such that $x\in B$, $\diam(B)<\delta$ and $Q_{\mu,\zeta_{o,\alpha}}(B)>t$. For every $z\in B\cap A$, we have 
$$
F_\delta^{\zeta_{o,\alpha}}(\mu,z)\ge Q_{\mu,\zeta^{o,\alpha}}(B)>t,
$$
so that $z\in \mathcal{S}_{t,\delta}$, hence $B\cap A\subset \mathcal{S}_{t,\delta}$. 
This shows that $\cS_{t,\delta}$ is open in the subspace topology of $A$ and in
particular it is Borel with respect to this topology.
\end{proof}

\begin{Rem}\label{rem Fb'}\rm
If we denote by $\cF_b'$ the family of closed balls with positive diameter,
then clearly $\cF_b'=(\cF_b')_{\mu,\zeta_{b,\alpha}}$. As a result, 
if $\cF_b'$ covers $A$ finely, then $(\cF_b)_{\mu,\zeta_{b,\alpha}}$ also covers $A$ finely
and the Federer density $\cs^\alpha(\mu,\cdot)$ of Theorem~\ref{th:spher Borel} is well defined.
For instance, if for every $x\in A$ there exists a sequence of closed
metric balls $\B_{x,k}$ (depending on $x$) with positive diameter and containing $x$, such that $\diam(\B_{x,k})\to 0$, then $\cF_b'$ covers $A$ finely and so does $(\cF)_{\mu,\zeta_{b,\alpha}}$.
\end{Rem}

%
%
%
%
%
%
%
%
%
\begin{The}[Area formula for the spherical measure I]\label{th:area formula sph Borel}
	Let $\mu$ be a measure over $X$, $\alpha>0$ and fix $A\subset X$.
	We assume the validity of the following conditions.
	\begin{enumerate}
		\item
		$X$ is a diametrically regular metric space.		
		\item 
		$\mu$ is both a regular measure and a Borel measure.
		\item
		$(\cF_b)_{\mu,\zeta_{b,\alpha}}$ covers $A$ finely.
		\item  
		$A$ is a Borel set.
		\item
		$A$ has a countable covering whose elements are open and have $\mu$-finite measure.
		\item
		The subset $\lbrace x\in A: \cs^\alpha(\mu,x)=0\rbrace$ is
		$\sigma$-finite with respect to $\cS^\alpha$.
		\item
		We have the absolute continuity $\mu\res A<<\cS^\alpha\res A$.
	\end{enumerate}
	Then $\cs^\alpha(\mu,\cdot):A\to[0,+\infty]$ is Borel and for every Borel set $B\subset A$ we have 
	\begin{equation}
	\mu(B)=\int_B\cs^\alpha(\mu,x)\,d\cS^\alpha(x).
	\end{equation}
\end{The}
\begin{proof}
Due to Theorem~\ref{th:spher Borel}, the density 
$\cs^\alpha(\mu,\cdot):A\to[0,+\infty]$ is Borel with respect to the subspace topology of $A$, hence by our assumptions it is also Borel. 
If $D_0=\B(x,r)$ is any closed ball, we consider its enlargement
\begin{equation}\label{d:Ball tau-enlargement}
	\widehat{D_0}=\bigcup\{D\in(\cF_b)_{\mu,\zeta_{b,\alpha}}:\,D\cap D_0\ne\emptyset\;\text{and}\;\diam(D)\le\tau\,\diam(D_0)\},
\end{equation}
according to \eqref{d:tau-enlargement}.
Then the closed ball $\widetilde{D_0}=\B(x,(1+\tau)\diam(D_0))$ contains $\widehat D_0$ 
and satisfies condition \eqref{ceta-estim} of Theorem~\ref{the:meastheoarea}.
All assumptions of Theorem~\ref{the:meastheoarea} are satisfied and this leads us to the conclusion.
\end{proof}	
%

%
%
%
%
%
%
%
%
%
%
\begin{The}[Area formula for the spherical measure II]\label{area formula spher meas}
	Let $\mu$ be a measure over $X$, $\alpha>0$ and fix $A\subset X$.
	We assume the validity of the following conditions.
	\begin{enumerate}
		\item 
		$X$ is a diametrically regular metric space.
		\item
		$\mu$ is a Borel regular measure.
		\item
		$(\cF_b)_{\mu,\zeta_{b,\alpha}}$ covers $A$ finely.
		\item  
		$A$ is ${\mathcal S^\alpha}$-measurable and $\sigma$-finite with respect to ${\mathcal{S}^\alpha}$.
		\item
		$A$ has a countable covering whose elements are open and have $\mu$-finite measure.
		\item
		We have the absolute continuity $\mu\res A<<\mathcal{S}^\alpha\res A$.
	\end{enumerate}
	Then  $\cs^\alpha(\mu,\cdot):A\to[0,+\infty]$ is $\cS^\alpha$-measurable,
	every $\cS^\alpha$-measurable set $B\subset A$ is also $\mu$-measurable and
	we have 
	\begin{equation} 
		\mu(B)=\int_B\cs^\alpha(\mu,x)\,d\cS^\alpha(x).
	\end{equation}
\end{The}
\begin{proof}
In view of Theorem~\ref{th:spher Borel}, the density $\cs^\alpha(\mu,\cdot):A\to[0,+\infty]$ is Borel with respect to the subspace topology of $A$. Being $A$ an $\cS^\alpha$-measurable set, then $\cs^\alpha(\mu,\cdot)$ is $\cS^\alpha$-measurable.
Arguing as in the proof of Theorem~\ref{th:area formula sph Borel}, we 
observe that condition \eqref{ceta-estim-II} of Theorem~\ref{theo-mea area form} is satisfied, again using the notion of enlargement of \eqref{d:tau-enlargement}. 
We can apply Theorem~\ref{theo-mea area form}, that concludes the proof.
\end{proof}

\begin{Rem}\rm 
Let $\cF_b'$ be the family of balls introduced in Remark~\ref{rem Fb'}.
Then in the assumptions of either Theorem~\ref{th:area formula sph Borel}
or Theorem~\ref{area formula spher meas} we actually have $(\cF_b)_{\mu,\zeta_{b,\alpha}}=\cF_b'$. As already pointed out in Remark~\ref{rem Fb'}, it holds $\cF_b'=(\cF_b')_{\mu,\zeta_{b,\alpha}}\subset (\cF_b)_{\mu,\zeta_{b,\alpha}}$. Conversely, any set of $(\cF_b)_{\mu,\zeta_{b,\alpha}}$
must differ from a point, due to the absolute continuity of $\mu\res A$
with respect to $\cS^\alpha\res A$, hence the opposite inclusion holds.
Thus, in both Theorems~\ref{th:area formula sph Borel} and \ref{area formula spher meas} the assumption (3) can be replaced by the condition that
$\cF_b'$ covers $A$ finely.
\end{Rem}

\bibliography{References}
\bibliographystyle{plain}

\end{document}